\documentclass{amsart}
\usepackage{mathrsfs}
\usepackage{stmaryrd,mathtools}
\usepackage{enumerate}
\usepackage{tikz-cd}
\usepackage[all]{xy}
\usepackage{aliascnt}
\usepackage{amscd}

\usepackage{shuffle}
\usepackage{ytableau}

\usepackage[colorlinks=true, linkcolor=blue, citecolor=blue]{hyperref}
\usepackage{fullpage}
\usepackage{verbatim}
\usepackage{amssymb}

\newtheorem{theorem}{Theorem}[section]

\newaliascnt{lemma}{theorem}
\newtheorem{lemma}[lemma]{Lemma}
\aliascntresetthe{lemma}

\newaliascnt{corollary}{theorem}
\newtheorem{corollary}[corollary]{Corollary}
\aliascntresetthe{corollary}

\newaliascnt{proposition}{theorem}
\newtheorem{proposition}[proposition]{Proposition}
\aliascntresetthe{proposition}

\newaliascnt{potato}{theorem}

\aliascntresetthe{potato}

\newaliascnt{definitionlemma}{theorem}

\aliascntresetthe{definitionlemma}

\newaliascnt{conjecture}{theorem}

\aliascntresetthe{conjecture}

\newaliascnt{question}{theorem}
\newtheorem{question}[question]{Question}
\aliascntresetthe{question}

\theoremstyle{definition}

\newaliascnt{definition}{theorem}

\aliascntresetthe{definition}

\newaliascnt{remark}{theorem}
\newtheorem{remark}[remark]{Remark}
\aliascntresetthe{remark}

\newaliascnt{example}{theorem}

\aliascntresetthe{example}

\newaliascnt{notation}{theorem}

\aliascntresetthe{notation}

\definecolor{darkblue}{rgb}{0.6,0,0.1}

\usepackage{tikz}
\usetikzlibrary{calc, shapes, backgrounds,arrows,positioning,plotmarks}
\tikzset{>=stealth',
  head/.style = {fill = white, text=black},
  plaque/.style = {draw, rectangle, minimum size = 10mm}, 
  pil/.style={->,thick},
  junct/.style = {draw,circle,inner sep=0.5pt,outer sep=0pt, fill=black}
  }

\newcommand{\ZZ}{\mathbb{Z}}
\newcommand{\QQ}{\mathbb{Q}}

\newcommand{\CC}{\mathbb{C}}

\newcommand{\bL}{\mathbb{L}}
\newcommand{\bG}{\mathbb{G}}
\newcommand{\bP}{\mathbb{P}}
\newcommand{\bA}{\mathbb{A}}

\newcommand{\cV}{\mathcal{V}}

\newcommand{\cF}{\mathcal{F}}
\newcommand{\cO}{\mathcal{O}}

\newcommand{\cH}{\mathcal{H}}
\newcommand{\cI}{\mathcal{I}}
\newcommand{\cL}{\mathcal{L}}

\newcommand{\cN}{\mathcal{N}}

\DeclareMathOperator{\Supp}{Supp}

\DeclareMathOperator{\reg}{reg}

\DeclareMathOperator{\Res}{Res}
\DeclareMathOperator{\coker}{coker}
\DeclareMathOperator{\Ind}{Ind}
\DeclareMathOperator{\pr}{pr}
\DeclareMathOperator{\Cone}{Cone}

\DeclareMathOperator{\Sing}{Sing}

\DeclareMathOperator{\Gr}{Gr}

\DeclareMathOperator{\Pic}{Pic}

\DeclareMathOperator{\Sym}{Sym}

\DeclareMathOperator{\Spec}{Spec}

\DeclareMathOperator{\id}{id}

\makeatletter
\@namedef{subjclassname@2020}{%
  \textup{2020} Mathematics Subject Classification}
\makeatother

\newif\ifhascomments \hascommentstrue
\ifhascomments
  \newcommand{\matt}[1]{{\color{red}[[\ensuremath{\spadesuit\spadesuit\spadesuit} #1]]}}
   \newcommand{\rosie}[1]{{\color{blue}[[\ensuremath{\clubsuit\clubsuit\clubsuit} #1]]}}
\else
  \newcommand{\matt}[1]{}
  \newcommand{\rosie}[1]{}
\fi

\title{$K$-regularity for finite quotient singularities}
\date{}
\author{Matthew Satriano}
\thanks{MS was partially supported by a Discovery Grant from the
  Natural Sciences and Engineering Research Council of Canada.}
\address[MS]{Department of Pure Mathematics, University
  of Waterloo, Waterloo ON N2L3G1, Canada}
\email{msatrian@uwaterloo.ca}

\author{Wanchun Shen}
\address[WS]{Department of Mathematics, Stanford University, Stanford, CA 94305, USA}
\email{wanchun@stanford.edu}

\begin{document}

\begin{abstract}
    We study $K$-regularity for complex varieties with finite quotient singularities. We prove that for such varieties, $K_1$-regularity is equivalent to smoothness whenever the singular locus is contained in an affine closed subvariety (e.g., affine varieties or those with isolated singularities). In contrast, we construct projective varieties with finite quotient singularities which are $K_m$-regular for every $m$ but are not local complete intersections. Our smoothness criterion is obtained by affirmatively answering Fogarty's 1988 question concerning differential forms on finite quotients.
\end{abstract}
\maketitle

\section{Introduction}
The interplay between homotopy invariance and regularity is a central theme in algebraic $K$-theory. For a regular scheme $X$, Quillen proved that the natural map
\[\varphi_{m,r}:K_m(X)\to K_m(X\times \mathbb{A}^r)\]
induced by the projection $X\times \mathbb{A}^r\to X$ is an isomorphism for all $m, r\ge 0$, see \cite[\S4 Theorem 3 Corollary 2 and \S7 Proposition 4.1]{Quillen1973}. A scheme $X$ is said to be \textit{$K_m$-regular} if $\varphi_{m,r}$ is an isomorphism for all $r\ge 0$. Thus, Quillen's theorem says that regular schemes are $K_m$-regular for all $m$. Vorst and Dayton-Weibel showed that $K_m$-regularity implies $K_{m-1}$-regularity \cite{Vorst1979,Dayton1980}.

This raises the natural question:~to what extent does homotopy invariance detect regularity? Vorst's Conjecture \cite{Vorst1979} asserts that if $X$ is affine of dimension $d$ and $X$ is $K_{d+1}$-regular, then it is, in fact, regular. For schemes of finite type over a field of characteristic zero, this was proved by Cortiñas-Haesemeyer-Weibel in \cite{Cortinas2007}; over perfect fields of positive characteristic, the corresponding statement was established by Kerz-Strunk-Tamme in \cite{Kerz2021}. In dimension one, Vorst's bound is sharp:~over a field of characteristic zero, a seminormal affine curve is $K_1$-regular \cite[Theorem~A]{VorstK1} and $K_2$-regularity is equivalent to regularity \cite[Theorem~3.6]{Vorst1979}.




In higher dimensions, however, recent work shows that for many classes of singularities, it suffices to check $K_m$-regularity for $m$ smaller than Vorst's bound. Haesemeyer-Weibel proved that $K_2$-regularity implies normality for affine varieties in characteristic zero \cite[Theorem~0.1]{HaesemeyerWeibel}, and that if an affine local complete intersection is $K_{p+1}$-regular, then it is regular in codimension $2p$ \cite[Theorem~0.3]{HaesemeyerWeibel}. Building on this work, the second author proved that an affine local complete intersection (lci) in characteristic zero is regular if and only if it is $K_2$-regular \cite[Theorem~E]{Shen2025}. Haesemeyer and Henderson-Walshe proved that $K_2$-regularity is equivalent to regularity for affine toric varieties \cite[Theorems~1.1 and~1.2]{Haesemeyer2026}, and for simplicial affine toric varieties, already $K_1$-regularity is equivalent to regularity \cite[Theorem~1.1]{Haesemeyer2026}.

In this paper, we investigate the extent to which homotopy invariance determines regularity for finite quotient singularities. Our first main result shows that for such varieties, $K_1$-regularity is enough to detect regularity provided the singular locus is contained in a closed affine subvariety:




\begin{theorem}\label{thm:finite-quotient-k1-regular-smooth}
Let $X$ be a complex variety with finite quotient singularities. Suppose $\Sing(X)$ is contained in an affine closed subvariety of $X$ (e.g., if $X$ is affine, or has isolated singularities). Then
\[
    X \text{ is $K_1$-regular}
    \quad\Longleftrightarrow\quad
    X \text{ is smooth}.
\]
\end{theorem}


The hypothesis on the singular locus $\Sing(X)$ cannot be removed. Indeed, we give examples of projective varieties $X$ with quotient singularities that are $K_m$-regular for every $m$, yet fail to be local complete intersections.

\begin{theorem}\label{thm:projective-finite-quotient-Km-regular}
    Let $A$ be an abelian variety and $\cL\in\Pic^0(A)$ a non-torsion line bundle. Let $d\geq3$ and
    \[
        X=\bP_A(\cO_A^{\oplus2}\oplus\cL)/\mu_d,
    \]
    where $\mu_d\subset\bG_m$ acts on $\cL$ by scalar multiplication. Then $X$ has finite quotient singularities and is $K_m$-regular for all $m$, but is not a local complete intersection.
\end{theorem}

\begin{remark}
    See \autoref{thm:Km-regular-G-torsor} and \autoref{cor:Km-regular-cone} for more general results ensuring $K_m$-regularity for all $m$. See \cite[Theorem 4.3]{HaesemeyerWeibelLineBundles} for another example that is $K_m$-regular for all $m$ (yet not normal).
\end{remark}

Note that if we take our abelian variety $A$ to be an elliptic curve, then $X$ has a $1$-dimensional singular locus consisting of finite quotient singularities and is $K_m$-regular for all $m$, but not lci.

\vspace{1em}

We prove \autoref{thm:finite-quotient-k1-regular-smooth} by answering the following question posed by Fogarty in 1988:

\begin{question}[Fogarty \cite{FOGARTY1988}]\label{q:Fogarty}
    Let $X=Y/G$, where $Y$ is smooth and $G$ is finite. Is it true that
    \[\Omega_X^1\twoheadrightarrow (\pi_*\Omega_Y^1)^G \implies \text{$X$ is smooth}?\]
\end{question}

Fogarty proved the case where $G$ is abelian, and Jamet's work \cite{Jamet} implies the result when $\dim X \le 2$. Building on Jamet's theorem and the Zassenhaus–Vincent–Wolf classification of fixed-point-free representations \cite{Zassenhaus1935, Vincent1947,Wolf2011,Stepanov}, we prove:

\begin{theorem}\label{thm:Fogarty's question}
    Let $X$ have finite quotient singularities. Then the following are equivalent:
    \begin{enumerate}
        \item $\Omega^1_X \to \Omega^{[1]}_X$ is surjective
        \item $X$ is smooth.
    \end{enumerate}
    In particular, Fogarty's \autoref{q:Fogarty} has a positive answer.
\end{theorem}

\autoref{thm:Fogarty's question} is also closely related to work of Greb and Rollenske \cite{GrebRollenske} on the cotorsion sheaf $\coker(\Omega^1_X\to\Omega^{[1]}_X)$. They studied the existence of cotorsion for mild singularities and gave broad classes of finite quotient singularities with non-zero cotorsion. \autoref{thm:Fogarty's question} therefore completes this picture for finite quotient singularities. In fact, as an immediate consequence of \autoref{thm:Fogarty's question}, we see the singular locus can be fully described as follows:

\begin{corollary}\label{cor:cotorsion}
    Let $X$ have finite quotient singularities. Then
    \[
        \Sing(X)=\Supp(\coker(\Omega^1_X\to\Omega^{[1]}_X)).
    \]
\end{corollary}

We additionally obtain the following alternative description of the singular locus. Let $\cN K_1$ denote the Zariski sheaf associated to the presheaf given by Bass' $NK$-groups $U\mapsto NK_1(U)$ (see e.g., \cite{vanderKallen86}). 
It is then an easy consequence of \autoref{thm:Fogarty's question} (see the last paragraph of the proof of \autoref{thm:finite-quotient-k1-regular-smooth}) that:

\begin{corollary}\label{cor:suppNK1}
    Let $X$ have finite quotient singularities. Then
    \[
        \Sing(X)=\Supp(\cN K_1).
    \]
\end{corollary}

We conclude by discussing an application of \autoref{thm:Fogarty's question} to higher Du Bois singularities, a class of singularities that has attracted growing interest in recent years \cite{MOPW, Jung2022, FL24a, FL24b, MP25, SVV, CDO26}. We say that a complex variety $X$ is \emph{strict $1$-Du Bois} if it has Du Bois singularities, and the natural morphism $\Omega_X^1\longrightarrow\underline{\Omega}_X^1$ is an isomorphism, where $\underline{\Omega}_X^1$ is the Du Bois complex on $X$. For a variety with finite quotient singularities, one has $\underline{\Omega}_X^1\simeq\Omega_X^{[1]}$, see \cite[Theorem 5.3]{DB81}. Thus, \autoref{thm:Fogarty's question} implies that 

\begin{corollary}
    Let $X$ be a complex algebraic variety with finite quotient singularities. Then
    \[
    X \text{ is strict-1-Du Bois}
    \quad\Longleftrightarrow\quad
    X \text{ is smooth}.
\]
\end{corollary}

\vspace{\medskipamount}

\noindent\emph{AI Disclosure}. 
The authors used AI models to assist with literature review, brainstorming, and testing examples. An AI model suggested the construction of $\eta$ in the proof of Proposition \ref{prop:typeI-prime-dim}, and we consulted AI for the case-by-case analysis of Proposition \ref{prop:ZVW->TypeI}. The proof of \autoref{thm:Km-regular-G-torsor} grew out of several brainstorming sessions with AI. All mathematical arguments and computations obtained with AI assistance were independently checked and verified by the authors, who take full responsibility for the arguments in the paper.

\subsection*{Acknowledgments} We thank Mihnea Popa for helpful comments on a preliminary version of this article, and Brad Dirks for discussions related to finite quotients that are local complete intersections.

\section{Background}

\subsection{Differential forms on finite quotients}
If $X$ is a complex variety and $F\subset \mathbb{C}$ is a subfield, $\Omega_{X/F}^p$ denotes the $p$-th K\"ahler differential of $X$ over $F$; if $F=\mathbb{C}$, we simply let $\Omega_X^p := \Omega_{X/\mathbb{C}}^p$. If $X$ is normal and $j\colon X_{\mathrm{reg}}\hookrightarrow X$ is the inclusion of the smooth locus, the sheaf of reflexive $p$-forms is
\[
\Omega_X^{[p]}:=j_*\Omega^p_{X_{\mathrm{reg}}}.
\]
In particular,
\[
\Omega_X^{[1]}\simeq(\Omega_X^1)^{\vee\vee}.
\]

Let a finite group $G$ act on a smooth complex variety $Y$ and let $\pi:Y\to X=Y/G$ be the quotient map. Pullback of differentials gives a natural map
\[
 d\pi:\Omega_X^1\longrightarrow(\pi_*\Omega_Y^1)^G.
\]
We then have canonical isomorphisms
\[
 (\pi_*\Omega_Y^1)^G
 \simeq j_*\Omega_{X_{\mathrm{reg}}}^1
 \simeq(\Omega_X^1)^{\vee\vee}
 =:\Omega_X^{[1]};
\]
see e.g., \cite[Lemma~1.8]{Steenbrink} and the introduction of \cite{Jamet}.  Under these identifications, $d\pi$ is the canonical map $\Omega^1_X\to\Omega^{[1]}_X$.

We also recall the description of the reflexive differentials in terms of the cdh-differentials. Let $X$ be a scheme over $F$, and let $a:X_{\mathrm{cdh}}\to X_{\mathrm{Zar}}$ be the change-of-topology map. We let
\[
    \Omega^p_{\mathrm{cdh},X/F}=Ra_*a^*\Omega^p_{X/F}.
\]
There are natural comparison maps $\Omega_{X/F}^p\to \Omega^p_{\mathrm{cdh},X/F}$ for every $p$. The complex $\Omega^p_{\mathrm{cdh},X/F}$ is isomorphic to the $p$-th Du Bois complex, see \cite[Lemma~2.1]{Shen2025}. We will therefore also write
\[
    \underline{\Omega}^p_{X/F}:=\Omega^p_{\mathrm{cdh},X/F}\quad\textrm{and}\quad\underline{\Omega}^p_{X}:=\underline{\Omega}^p_{X/\CC}
\]

Recall that a variety $X$ is said to have finite quotient singularities if every $x\in X$ has an \'etale neighborhood isomorphic to a quotient $U/G$, where $U$ is smooth and $G$ is finite. Since finite quotient singularities are rational, $\cH^0\underline{\Omega}^1_X$ is reflexive; see \cite[Remark~2.5]{Shen2025}. Since the comparison map is an isomorphism on the smooth locus, it follows that for a complex variety with finite quotient singularities,
\[
\mathcal H^0\underline{\Omega}^1_X
\simeq
j_*\Omega^1_{X_{\mathrm{reg}}}
=
\Omega_X^{[1]}.
\]
Under these identifications, the comparison map
\[
\Omega_X^1\longrightarrow\cH^0\underline{\Omega}^1_X
\]
agrees with the map $d\pi$. Indeed, the two maps agree over $X_{\mathrm{reg}}$, and their common target is torsion-free.

\subsection{$K$-regularity}
A scheme $X$ is called \textit{$K_m$-regular} if the natural maps
\[
 K_m(X)\longrightarrow K_m(X\times\mathbb A^r)
\]
are isomorphisms for every $r\geq0$. By \cite[\S4 Theorem 3 Corollary 2 and \S7 Proposition 4.1]{Quillen1973}, a regular scheme is $K_m$-regular for every $m$. Furthermore, $K_m$-regularity implies $K_{m-1}$-regularity \cite{Vorst1979, Dayton1980}. One way to measure the failure of $K_m$-regularity is through \textit{Bass's $NK$-groups}, given by
\[
 NK_m(X)=\operatorname{coker}\!\left(
 K_m(X)\longrightarrow K_m(X\times\mathbb A^1)
 \right).
\]
By \cite[Theorem 0.1]{CHWW}, $X$ is $K_m$-regular if and only if $NK_q(X)=0$ for every $q\leq m$.

The relation between the Bass $NK$-groups and differential forms is given by the comparison between derived K\"ahler differentials and cdh differentials \cite{CHSW,Cortinas2007,CHWW}. Let $F$ be a field of characteristic zero, and write
\[
L^p_{X/F}:=L\bigwedge^pL_{X/F}.
\]
We define
\[
C^p_{X/F}:=
\operatorname{Cone}\left(
L^p_{X/F}\longrightarrow
\underline{\Omega}^p_{X/F}
\right).
\]
When $F=\mathbb C$, we simply write $L_X^p$ and $C_X^p$.

For an affine scheme $X$ over $\mathbb Q$, the description of the Bass $NK$-groups in \cite{CHWW}, in the form recorded in \cite[Theorem~2.8 and Remark~2.9]{Shen2025}, gives
\[
X\text{ is }K_m\text{-regular}
\quad\Longleftrightarrow\quad
\mathcal H^j(C^p_{X/\mathbb Q})=0
\quad\text{for all }p,j\text{ with }j-p\geq-m.
\]

For projective varieties there is also a global criterion for $K$-regularity. If $X$ is a complex projective variety, then \cite[Theorem~7.1]{Shen2025} shows $K_m$-regularity is equivalent to the comparison maps
\[
H^i(X,L_X^p)\longrightarrow H^i(X,\underline{\Omega}_X^p)
\]
being isomorphisms for $i-p\geq-m+1$. Thus, if $R\Gamma(X,C_X^p)=0$ for all $p\geq0$, then $X$ is $K_m$-regular for all $m$.

\subsection{Zassenhaus-Vincent-Wolf classification}
A complex representation $V$ of a finite group $G$
is called \emph{fixed-point free} if
\[
 V^g=0\qquad\text{for every }g\neq1;
\]
equivalently, $1$ is not an eigenvalue for any non-trivial element of $G$. If $G\neq1$ and $\dim V\geq2$, the quotient $V/G$ is then singular only at the image of the origin.  Conversely, after quotienting out by the subgroup generated by pseudo-reflections, the representation
defining an isolated quotient singularity is fixed-point free.

The classification of finite groups admitting fixed-point free complex representations started with work of Zassenhaus \cite{Zassenhaus1935} and Vincent \cite{Vincent1947}, and was completed by Wolf \cite{Wolf2011}; see the introduction of \cite{Stepanov} for more details. We use the classification in the form given by Stepanov \cite[\S3]{Stepanov}. In Theorems 3.1 and 3.2 of \cite{Stepanov}, Stepanov gives the six families of groups admitting fixed-point free complex representations, and in Theorem 3.6 and Table 2 of \cite{Stepanov}, he lists the irreducible fixed-point free complex representations of these groups.

\section{Answering Fogarty's question:~proof of \autoref{thm:Fogarty's question}}

We prove \autoref{thm:Fogarty's question} in three steps. First we reduce to the case of fixed-point free representations, then using the Zassenhaus--Vincent--Wolf classification we reduce to certain Type I representations, and lastly we prove \autoref{thm:Fogarty's question} for these representations. Each of these steps is handled in a different subsection.

\subsection{Reduction to the case of fixed-point free representations}

The following result allows us to look \'etale locally.

\begin{lemma}\label{lemma:surj-etale-local}
    Let $f: X\to Y$ be a map between normal varieties that is \'etale at $x\in X$. Then the following are equivalent:
    \begin{enumerate}
        \item The map $\Omega_X^1\to \Omega_X^{[1]}$ is surjective at $x$,
        \item The map $\Omega_Y^1\to \Omega_Y^{[1]}$ is surjective at $f(x)$.
    \end{enumerate}
\end{lemma}
\begin{proof}

    Since K\"ahler differentials commute with \'etale base change \cite[Tag 0FLV]{stacks-project}, we have $f^*\Omega_Y^1\simeq\Omega_X^1$. Since $f$ is flat, $f^*\Omega_Y^{[1]}$ is reflexive by \cite[Proposition 1.8]{Hartshorne1980}. Over $f^{-1}(Y_{\reg})=X_{\reg}$ it agrees with $\Omega^1_{X_{\reg}}$, and so $f^*\Omega_Y^{[1]}\simeq\Omega_X^{[1]}$. We therefore have a commutative diagram
    \[\begin{tikzcd}
        f^*\Omega_{Y,f(x)}^1 \arrow[r] \arrow[d, "\simeq"'] &
f^*\Omega_{Y,f(x)}^{[1]} \arrow[d, "\simeq"] \\
\Omega_{X,x}^1 \arrow[r] &
\Omega_{X,x}^{[1]}
    \end{tikzcd}\]
    where the vertical maps are isomorphisms. Since pullback preserves cokernels, 
    \[
        \coker(\Omega^1_{X,x}\to\Omega^{[1]}_{X,x})\,\simeq\, \cO_{X,x}\otimes_{\cO_{Y,{f(x)}}}\coker(\Omega^1_{Y,f(x)}\to\Omega^{[1]}_{Y,f(x)}).
    \]
    Since $\cO_{Y, f(x)}\to\cO_{X,x}$ is an \'etale (hence flat) map of local rings, it is faithfully flat by \cite[Tag 00HR]{stacks-project}. Thus, $\Omega_{Y, f(x)}^1\to \Omega_{Y,f(x)}^{[1]}$ is surjective if and only if $\Omega_{X,x}^1\to \Omega_{X,x}^{[1]}$ is.
\end{proof}

The following is the main result of this subsection.

\begin{proposition}\label{prop:reduction}
To prove \autoref{thm:Fogarty's question}, it suffices to show that
    \[
        C_G(V):=\coker\left(\Omega_{V/G}^1\longrightarrow(\pi_*\Omega_V^1)^G\right)\neq0
    \]
for every fixed-point free representation $V$ of a non-trivial finite group $G$ with $\dim V\geq2$.
\end{proposition}
\begin{proof}
    To prove \autoref{thm:Fogarty's question}, we may look \'etale locally where $X$ is a global quotient by a finite group, i.e., we have $\pi\colon Y\to Y/G=:X$ with $Y$ smooth, $G$ a finite group, and $X$ singular. Covering $X$ by open affines, we may assume both $X$ and $Y$ are affine. Let $Z\subset\Sing(X)$ be a maximal-dimensional irreducible component of the singular locus of $X$. Then, choosing a Luna slice at a general point $\pi(y)\in Z$ and using that $\Omega^1$ and $\Omega^{[1]}$ commute with \'etale base change, the proof of Theorem 4.1 in \cite{Jamet} (see the paragraph on page 164 above the Conclusions section) shows that we may replace $X$ by $V/G_y$, where $T_yY=(T_yY)^{G_y}\oplus V$ as $G_y$-representations; furthermore, $V/G_y$ has an isolated singularity at the origin. By the first sentence of the proof of Theorem 5.1 in \cite{Jamet} (which does not use the assumption that the group is abelian), we are reduced to the case of a small representation $W$ of $H$ with an isolated singularity at $0$, i.e., $H$ has no non-trivial pseudo-reflections.
    
    We next show $W/H$ is smooth away from the origin. Since $W$ is a representation, the $H$-action commutes with the $\bG_m$-action given $\lambda\cdot w=\lambda w$. Thus, the $\bG_m$-action descends to $W/H$. If $z\in W/H$ is singular, we then see $\lambda z$ is as well for all $\lambda$. Thus $0$ is in the closure of $\bG_m z$, and so $z=0$. This shows $\Sing(W/H)=\{0\}$.
    
    Since $W/H$ is singular, we see $\dim W\geq2$. Lastly, every $0\neq w\in W$ maps to a smooth point, so the Chevalley--Shephard--Todd Theorem (see, e.g., \cite[Remark 11.7]{Etingof2024}) tells us that $H_w$ is generated by pseudo-reflections. If $H_w\neq1$, then $H$ contains a non-trivial pseudo-reflection which is a contradiction. Thus, $W$ is a fixed-point free representation.
\end{proof}

\subsection{Reducing the case of fixed-point free representations to certain Type I representations}

For a $G$-representation $V$, with quotient map $\pi:V\to V/G$, write
\[
 C_G(V)=
 \operatorname{coker}\!\left(
 \Omega_{V/G}^1\longrightarrow(\pi_*\Omega_V^1)^G
 \right)
\]
as in \autoref{prop:reduction}. To prove \autoref{thm:Fogarty's question}, we must consider all types in the Zassenhaus--Vincent--Wolf classification. The following results will be useful in reducing the general case to a representation of a Type I group of prime dimension $p\neq 2$.

\begin{theorem}[Jamet]\label{lem:surface}
Let $V$ be a nontrivial two-dimensional fixed-point free representation of
$G$.  Then $C_G(V)\neq0$.
\end{theorem}
\begin{proof}
This is immediate from \cite[Theorem 4.1]{Jamet}.
\end{proof}

\begin{lemma}\label{fixed-pt-free->singular-quotient}
Let $V$ be a fixed-point free $G$-representation. If $V/G$ is smooth, then $G$ is trivial or $\dim V=1$.
\end{lemma}
\begin{proof}
    Since $G$ acts on $V$ without fixed points, it is a faithful representation. If $V/G$ is smooth, then by the Chevalley--Shephard--Todd Theorem \cite{Chevalley,ShephardTodd}, $G$ is generated by pseudo-reflections. Since every pseudo-reflection $g$ fixes a hyperplane, we see either $g=1$ or the hyperplane must be trivial, namely $0$. In the latter case, $\dim V=1$. If $\dim V>1$, then $G$ is generated by pseudo-reflections, all of which are trivial, hence $G=1$.
\end{proof}

\begin{lemma}\label{l:induction-properties}
Let $H\subset G$ be a subgroup, $W$ an $H$-representation, and $V=\Ind_H^GW$. 
\begin{enumerate}
    \item\label{redInd} If $V$ is irreducible, then $W$ is as well.
    \item\label{fpfInd} If $V$ is fixed-point free, then $W$ is as well.
\end{enumerate}
\end{lemma}
\begin{proof}
    For (\ref{redInd}), if $W=W_1\oplus W_2$ with $W_i$ non-trivial $H$-subrepresentations, then $V=\Ind_H^GW_1\oplus\Ind_H^GW_2$. For (\ref{fpfInd}), note that $W$ is the $H$-subrepresentation of $\Res^G_HV$ given by the identity coset. Thus, if $h\neq1$ fixes $0\neq w\in W$, then we see $h$ is also a non-trivial element of $G$ which fixes $w\in\Res^G_HV$, which is equal to $V$ as a vector space.
\end{proof}

\begin{lemma}[Direct summands]\label{lem:summand}
If $V=W\oplus Z$ as $G$-representations and $C_G(W)\neq0$, then
$C_G(V)\neq0$.
\end{lemma}
\begin{proof}
Let $p\colon V/G\to W/G$ and $i\colon W/G\to V/G$ be the natural maps. We have a commutative diagram where the rows are exact
\[
    \xymatrix{
        i^*p^*\Omega^1_{W/G}\ar[r]\ar[d] & i^*p^*(\Omega^1_W)^G\ar[r]\ar[d] & i^*p^*C_G(W)\ar[r]\ar[d] & 0\\
        i^*\Omega^1_{V/G}\ar[r]\ar[d] & i^*(\Omega^1_V)^G\ar[r]\ar[d] & i^*C_G(V)\ar[r]\ar[d] & 0\\
        \Omega^1_{W/G}\ar[r] & (\Omega^1_W)^G\ar[r] & C_G(W)\ar[r] & 0
    }
\]
Since $i^*p^*=\id$, the composition of the vertical arrows in any column is the identity map, and hence, $C_G(W)=i^*p^*C_G(W)\to i^*C_G(V)$ is injective. Thus, if $C_G(W)\neq0$, then $i^*C_G(V)\neq0$. Hence, $C_G(V)\neq0$.
\end{proof}

\begin{lemma}[Induction]\label{lem:induction}
Let $H\leq G$, let $W$ be an $H$-representation, and let
$V=\operatorname{Ind}_H^G W$.  If $C_H(W)\neq0$, then $C_G(V)\neq0$.
\end{lemma}
\begin{proof}
    The inclusion $W\subset V$ is $H$-equivariant, so we have an induced map $W/H\to V/G$ making the diagram
    \[
        \xymatrix{
            W\ar[d]_-{\pi_W}\ar[r] & V\ar[d]^-{\pi_V}\\
            W/H\ar[r]^-{i} & V/G
        }
    \]
    commute. We then obtain a commutative diagram where the rows are exact
    \[
        \xymatrix{
            i^*\Omega^1_{V/G}\ar[r]\ar[d] & i^*((\pi_V)_*\Omega^1_V)^G\ar[d]^-{\alpha}\ar[r] & i^*C_G(V)\ar[r]\ar[d] & 0\\
            \Omega^1_{W/H}\ar[r] & ((\pi_W)_*\Omega^1_W)^H\ar[r] & C_H(W)\ar[r] & 0
        }
    \]
    Upon showing $\alpha$ is surjective, we are done; indeed, we would then have $i^*C_G(V)\to C_H(W)$ is surjective so $i^*C_G(V)\neq0$ and hence $C_G(V)\neq0$.

    To prove $\alpha$ is surjective, it suffices to show $(\Omega^1_{k[V]})^G\to(\Omega^1_{k[W]})^H$ is surjective. Let $V=\bigoplus_{gH\in G/H}gW$, $p_{gH}\colon V\to gW$ be the projection, and $t_{gH}\colon W\xrightarrow{\simeq} gW$ be the isomorphism given by $t_{gH}(w)=gw$. If $\eta\in(\Omega^1_{k[W]})^H$, let $\eta_{gH}:=p_{gH}^*(t_{gH}^{-1})^*\eta$ and let $\eta':=\sum_{gH\in G/H}\eta_{gH}$. Then $\eta'\in (\Omega^1_{k[V]})^G$ and we see $i^*\eta'=\eta$. Indeed, $i^*p_{H}^*=\id$ and $i^*p_{gH}^*=0$ for $gH\neq H$. This proves surjectivity of $(\Omega^1_{k[V]})^G\to(\Omega^1_{k[W]})^H$, hence also of $\alpha$.
\end{proof}

\begin{lemma}[Reduction to the irreducible case]\label{lem:general->irrep}
Let $G\neq1$ and let $V$ be a fixed-point free $G$-representation with $\dim V\geq2$. If $C_G(V)=0$, then there is an irreducible fixed-point free $G$-subrepresentation $W\subset V$ such that $\dim W\geq2$ and $C_G(W)=0$.
\end{lemma}
\begin{proof}
    We may write $V=\bigoplus_{i=1}^r V_i$ with each $V_i$ an irreducible $G$-representation. Since $V$ is fixed-point free, each $V_i$ is as well. We claim some $\dim V_i>1$. If instead $\dim V_i=1$ for all $i$, then since $\dim V\geq2$, we see $r\geq2$. Then \autoref{lem:summand} shows $C_G(V_1\oplus V_2)=0$, however this contradicts \autoref{lem:surface}. Having now shown $\dim V_j\geq2$ for some $j$, we see from \autoref{lem:summand} that $C_G(V_j)=0$.
\end{proof}

As mentioned earlier, the Zassenhaus--Vincent--Wolf classification gives six types of groups and lists all their irreducible fixed-point free representations; we use the notation of Theorems~3.1, 3.2 and the list preceding Theorem~3.6 in \cite{Stepanov}.

\begin{lemma}\label{lem:TypeI-induction}
Following the notation in \cite[p.~822]{Stepanov}, let $K=\langle A,B\rangle$ be a group of Type I and let $V=\pi_{k,l}$ be an irreducible fixed-point free $K$-representation of dimension $d$. Then
\begin{enumerate}
    \item\label{typeIcharacter} There is a subgroup $K_0\subseteq K$ and a character $\chi$ of $K_0$ such that $V\simeq\Ind_{K_0}^K\chi$.

    \item\label{typeIdimp} If $d\geq2$ and $p$ is a prime factor of $d$, then there is a subgroup $K_p$ for which $K_0\subset K_p\subset K$ and an irreducible fixed-point free $p$-dimensional representation $W_p$ of $K_p$ such that $V\simeq\Ind_{K_p}^K W_p$. Moreover, if $p$ is odd, then $W_p$ is Type I.

    \item\label{boxtimes->2dim} Let $H$ be a finite group and let $\tau$ be a $2$-dimensional $H$-representation. If $V\boxtimes\tau$ is a fixed-point free representation of $K\times H$, then it is of the form $\Ind_{K_0\times H}^{K\times H} W$ where $W$ is a $2$-dimensional fixed-point free representation.
\end{enumerate}
\end{lemma}

\begin{proof}
Let $K_0:=\langle A,B^d\rangle$. Since $r^d\equiv1\pmod m$, we see $A$ and $B^d$ commute. Let $e_0,\dots,e_{d-1}$ be the basis used to write the matrices for $\pi_{k,l}$ in \cite[p.~822]{Stepanov}. Note that $K_0\subset K$ preserves the line $L=\CC e_0$ and hence we obtain a character $\chi:=\chi_{k,l}\colon K_0\to\CC^*$ given by
\[
    \chi(A)=e^{2\pi i k/m} \quad\textrm{and}\quad \chi(B^d)=e^{2\pi i ld/n}.
\]
Furthermore by Frobenius reciprocity, the inclusion of $K_0$-representations $L\subset\Res^K_{K_0}V$ yields a map $\Ind_{K_0}^K\chi\to V$ of $K$-representations. Since $V$ is irreducible, the map is surjective. Since both representations have the same dimension, the map is therefore an isomorphism, thereby proving (\ref{typeIcharacter}).

We next turn to (\ref{typeIdimp}). Let $s=d/p$ and $K_p:=\langle A,B^s\rangle$. Since $B^d=(B^s)^p$, we see then that $K_0\subseteq K_p\subseteq K$. Note that $K_0$ is normal in $K_p$ of index $p$, so we see $W_p:=\Ind_{K_0}^{K_p}\chi$ has dimension $p$ and 
\[
    V\simeq\Ind_{K_0}^K\chi\simeq\Ind_{K_p}^K W_p.
\]
Note that $W_p$ is irreducible and fixed-point free by \autoref{l:induction-properties}. Finally, if $p$ is odd, then since $W_p$ is an odd-dimensional irreducible fixed-point free representation, it follows from \cite[Theorem 3.6]{Stepanov} that it must be of Type I.

Lastly, to prove (\ref{boxtimes->2dim}), using (\ref{typeIcharacter}) we may write $V\simeq\Ind_{K_0}^K\chi$ for a character $\chi$. Then we see
\[
    V\boxtimes\tau \simeq (\Ind_{K_0}^K\chi)\boxtimes\tau \simeq \Ind_{K_0\times H}^{K\times H}(\chi\boxtimes\tau).
\]
We see $\chi\boxtimes\tau$ is a $2$-dimensional representation. It remains to show that $\chi\boxtimes\tau$ is fixed-point free. Since $\chi\boxtimes\tau$ is a $(K_0\times H)$-stable subspace of $\Res^{K\times H}_{K_0\times H}(V\boxtimes\tau)$, if an element of $K_0\times H$ fixed a non-zero vector of $\chi\boxtimes\tau$, it would also fix a nonzero vector of $V\boxtimes\tau$; this is a contradiction since $V\boxtimes\tau$ is fixed-point free. Hence, $\chi\boxtimes\tau$ is fixed-point free.
\end{proof}

The following is the key result allowing us to reduce to Type I representations.

\begin{proposition}\label{prop:ZVW->TypeI}
    If $V$ is an irreducible fixed-point free $G$-representation with $\dim V\geq2$, then there is a subgroup $H\subset G$ and an irreducible fixed-point free $H$-representation $W$ such that $V=\Ind^G_H W$ and either $\dim W=2$ or $W$ is a prime-dimensional Type I representation.
\end{proposition}

\begin{proof}
We follow the notation on p.~822--826, Theorem 3.6, and Table 2 of \cite{Stepanov}. We successively run through Type I--VI representations and prove the desired result by induction. Note that by \autoref{l:induction-properties}, upon showing $V=\Ind_H^GW$, we automatically know $W$ is irreducible and fixed-point free.

\vspace{1em}
\noindent\textbf{Type I.} By \autoref{lem:TypeI-induction}(\ref{typeIdimp}), we may choose $H=K_p$ and $W=W_p$.

\vspace{1em}
\noindent\textbf{Type II.}
Following the notation on \cite[p.~822]{Stepanov}, an irreducible fixed-point free representation $V$ of Type II is given by $\alpha_{k',l'}$, which has dimension $2d$ and is of the form $\Ind_{\langle A,B\rangle}^G\pi_{k',l'}$ with $\pi_{k',l'}$ a $d$-dimensional representation. If $d=1$, then $\dim V=2$ so we may choose $H=G$ and $W=V$; otherwise, we are reduced to the case of Type I representations of dimension at least $2$.

\vspace{1em}
\noindent\textbf{Type III.}
There are 3 different subcases we must consider. First, if $9\nmid n$, then $G=\langle A,B^3\rangle\times T^*$ and $V=\pi_{k,l}\boxtimes\tau$, where $\langle A,B^3\rangle$ is of Type I, $T^*$ is the binary tetrahedral group, and $\tau$ is the unique irreducible fixed-point free representation of $T^*$. By \cite[Lemma~2.15]{Stepanov}, $\dim\tau=2$. Therefore, by \autoref{lem:TypeI-induction}(\ref{boxtimes->2dim}), $V$ has the desired form.

Next, if $9\mid n$ and $3\nmid d$, then $G=\langle A,B^{3^v}\rangle\times T_v^*$ and $V=\pi_{k,l}\boxtimes\tau_j$ with $\dim\tau_j=2$ by \cite[Lemma~2.15]{Stepanov}. So, by \autoref{lem:TypeI-induction}(\ref{boxtimes->2dim}), $V$ again has the desired form.

Lastly, if $3\mid d$, then $K=\langle A,B^d,P,Q\rangle$ is a normal subgroup of index $d$ and $V=\Ind_K^G W$. Since $\dim V=2d$, we see $\dim W=2$.

\vspace{1em}
\noindent
\textbf{Type IV.} 
There are again several cases we must consider. Using the notation on \cite[p.~823]{Stepanov}, let $G':=\langle A,B,P,Q\rangle$, which is a Type III subgroup of $G$. 

We consider Case 1, which assumes $9\nmid n$. In Subcase 1(a), $G=\langle A,B^3\rangle\times O^*$ and $V=\pi_{k,l}\boxtimes o_j$, where $\pi_{k,l}$ is Type I and $o_j$ is an irreducible fixed-point free representation of the binary octahedral group $O^*$. By \cite[Lemma~2.16]{Stepanov}, $\dim o_j=2$. By \autoref{lem:TypeI-induction}(\ref{boxtimes->2dim}), $V$ has the form we want.

In Subcase 1(b), our representation $\gamma_{k',l'}$ is $4d$-dimensional and is induced from the $2d$-dimensional Type III representation $\nu_{k',l'}$ of $G'$. By the Type III case, we know then that $\nu_{k',l'}\simeq\Ind_H^{G'}W$ with $W$ a $2$-dimensional representation. So, $\gamma_{k',l'}\simeq\Ind_{G'}^G\nu_{k',l'}\simeq\Ind_H^GW$.

We now turn to Case 2, which assumes $9\mid n$ and $3\nmid d$. In Subcase 2(a), $G=\langle A,B^{3^v}\rangle\times O_v^*$ and $V=\pi_{k,l}\boxtimes o_j$. By \cite[Lemma 2.16]{Stepanov}, $o_j\simeq\Ind_{T_v^*}^{O_v^*}\tau_j$, where $\tau_j$ is $2$-dimensional. Letting $K=\langle A,B^{3^v}\rangle$ and $\pi_{k,l}=\Ind_{K_0}^K\chi$ as in \autoref{lem:TypeI-induction}(\ref{typeIcharacter}), we see
\[
    \pi_{k,l}\boxtimes o_j\simeq\Ind_{K\times T_v^*}^{K\times O_v^*}(\pi_{k,l}\boxtimes\tau_j)\simeq\Ind_{K_0\times T_v^*}^G(\chi\boxtimes\tau_j).
\]
Then $\chi\boxtimes\tau_j$ is $2$-dimensional.

Next, in Subcase 2(b), our representation $\gamma_{k',l',j}$ is induced from the Type III representation $\nu_{k',l',j}$ of $G'$, so the Type III case and induction in stages reduce us to a $2$-dimensional fixed-point free representation.

Lastly, Case 3 assumes $3\mid d$. Here the representation $\eta_{k',l'}$ is induced from the Type III representation $\mu_{k',l'}$ of $G'$. The Type III case then again reduces this to a $2$-dimensional fixed-point free representation.

\vspace{1em}
\noindent\textbf{Type V.}
By Theorem~3.2 and p.~826 of \cite{Stepanov}, $G=K\times I^*$ where $K$ is of Type I, and $V=\pi_{k,l}\boxtimes\iota_j$. By \cite[Lemma~2.17]{Stepanov}, the irreducible fixed-point free representation $\iota_j$ of the binary icosahedral group $I^*$ has dimension $2$. Then by \autoref{lem:TypeI-induction}(\ref{boxtimes->2dim}), $V$ has the desired form.

\vspace{1em}
\noindent\textbf{Type VI.}
By \cite[p.~826]{Stepanov}, every Type VI representation $V$ is induced from an irreducible fixed-point free representation of Type V. Thus, we are done by induction.
\end{proof}

\subsection{Proof of \autoref{thm:Fogarty's question}}

Recall that a Type I group is given by
\[
 G=\langle A,B\mid A^m=B^n=1,\ BAB^{-1}=A^r\rangle
\]
where $m,n\geq1$ and
\begin{enumerate}
 \item[(1)] $\gcd(n(r-1),m)=1$,
 \item[(2)] $r^n\equiv1\pmod m$,
 \item[(3)] if $p$ is a prime dividing the order $d$ of $r$ in $(\ZZ/m)^*$, then $p$ divides $n/d$.
\end{enumerate}
The first two conditions are stated in \cite[Table 1]{Stepanov} and the last condition is from \cite[Theorem 3.1]{Stepanov}. The irreducible fixed-point free representations of $G$ are given by
\[
 \pi_{k,l}(A)=
 \begin{pmatrix}
  e^{2\pi i k/m} & 0 & \cdots & 0\\
  0 & e^{2\pi i kr/m} & \cdots & 0\\
  \vdots & \vdots & \ddots & \vdots\\
  0 & 0 & \cdots & e^{2\pi i kr^{d-1}/m}
 \end{pmatrix},\quad\quad
 \pi_{k,l}(B)=
 \begin{pmatrix}
  0 & 1 & \cdots & 0\\
  \vdots & \vdots & \ddots & \vdots\\
  0 & 0 & \cdots & 1\\
  e^{2\pi i ld/n} & 0 & \cdots & 0
 \end{pmatrix}
\]
where $\gcd(k,m)=\gcd(l,n)=1$.

\begin{proposition}\label{prop:typeI-prime-dim}
    Let $V$ be the irreducible fixed-point free representation $\pi_{k,l}$ of Type I. If $\dim V$ is prime, then $C_G(V)\neq0$.
\end{proposition}

%


\begin{proof}
Consider the $G$-representation $\pi_{k,l}$ as described above, where the matrices are written with respect to the basis $x_0,\ldots,x_{p-1}$ of $V$. Since every prime factor of the order of $r$ in $(\ZZ/m)^*$ divides $n/p$, we see $s:=n/p^2\in\ZZ$. Let $\zeta:=e^{2\pi il/p}$ and
\[
 \eta=
 \sum_{j=0}^{p-1} \left(\zeta^{-j} \prod_{0\leq i\leq p-1}x_i^s\right)\frac{dx_j}{x_j}.
\]
We claim that $0\neq\eta\in C_G(V)$.

We first show $\eta$ is $G$-invariant. To see this, first recall $\gcd(p(r-1),m)=1$ and so $\gcd(r-1,m)=1$. Then using that $0\equiv r^p-1\equiv (r-1)(1+r+\dots+r^{p-1})\pmod{m}$, we see $1+r+\dots+r^{p-1}\equiv0\pmod{m}$. Hence, letting $\zeta'=e^{2\pi i/m}$, we have
\[
    A\cdot \prod_i x_i=(\zeta')^{k(1+r+\dots+r^{p-1})}\prod_i x_i=\prod_i x_i.
\]
It follows that $A\cdot\eta=\eta$.

To see that $B\cdot\eta=\eta$, let $\beta=e^{2\pi i\ell p/n}$ and observe that $B\cdot \prod_i x_i=\beta\prod_i x_i$. So, $B\cdot\prod_i x_i^s=\zeta\prod_i x_i^s$. Hence,
\[
    B\cdot\eta=\sum_{j\in\ZZ/p} \left(\zeta^{-(j-1)} \prod_{0\leq i\leq p-1}x_i^s\right)\frac{dx_{j-1}}{x_{j-1}} = \eta.
\]

It remains to show $\eta$ is not in the image of $\Omega_{V/G}^1$. For this, let 
\[
 E=\sum_{j=0}^{p-1}x_j\frac{\partial}{\partial x_j}.
\]
Applying the contraction map $\iota_E$ induced by $E$, we have
\[
 \iota_E(\eta)=
 \prod_{i=0}^{p-1}x_i^s\sum_{j=0}^{p-1}\zeta^j=0.
\]
Now suppose $\eta$ were of the form $\sum_i a_i db_i$ with $a_i,b_i\in\CC[V]^G$; we may assume the $a_i$ and $b_i$ are homogeneous. Since $\pi_{k,l}(B^p)=e^{2\pi i lp/n}\id_V$, we see that every homogeneous element $h\in\CC[V]^G$ of degree $D$ satisfies $h=B^p\cdot h=e^{2\pi i lDp/n}h$, and so $n/p$ divides $D$. Since $\eta$ has degree $n/p$, we see $\deg a_i=0$ and $\deg b_i=n/p$, so $\eta=df$ for some homogeneous $f\in(\CC[V]^G)_{n/p}$. As a result,
\[
 0=\iota_E(\eta)=\iota_E(df)=\frac np f.
\]
It follows that $f=0$ and hence $\eta=0$, a contradiction. Thus, $\eta$ represents a nonzero element of $C_G(V)$.
\end{proof}

\begin{proof}[{Proof of \autoref{thm:Fogarty's question}}]
    By \autoref{prop:reduction}, it suffices to show $C_G(V)\neq0$ for every fixed-point free $G$-representation $V$ with $\dim V\geq2$ and $G\neq1$. By \autoref{lem:general->irrep}, we may assume $V$ is an irreducible representation. Then by \autoref{prop:ZVW->TypeI}, \autoref{lem:surface}, and \autoref{lem:induction}, we may assume $V$ is a prime-dimensional Type I representation. Then \autoref{prop:typeI-prime-dim} finishes the proof.
\end{proof}

\section{$K_1$-regularity and affine singular loci:~proof of \autoref{thm:finite-quotient-k1-regular-smooth}}

\begin{lemma}\label{l:C->Q}
Let $X$ be a complex variety with Du Bois singularities. Then the
natural map
\[
\Omega^1_{X/\mathbb C}\longrightarrow
\cH^0\underline{\Omega}^1_{X/\mathbb C}
\]
is surjective if and only if
\[
\Omega^1_{X/\mathbb Q}\longrightarrow
\cH^0\underline{\Omega}^1_{X/\mathbb Q}
\]
is surjective.
\end{lemma}

\begin{proof}
We claim that we have the following morphism of exact triangles
\[
\begin{CD}
\Omega^1_{\mathbb C/\mathbb Q}\otimes_{\mathbb C}\mathcal O_X
    @>>> L_{X/\mathbb Q}
    @>>> L_{X/\mathbb C} \\
@VVV @VVV @VVV \\
\Omega^1_{\mathbb C/\mathbb Q}\otimes_{\mathbb C}
    \underline{\Omega}^0_{X/\mathbb C}
    @>>> \underline{\Omega}^1_{X/\mathbb Q}
    @>>> \underline{\Omega}^1_{X/\mathbb C}.
\end{CD}
\]
where the top row is the transitivity triangle for the cotangent complex \cite[Tag~08QX]{stacks-project}. To see that $\Omega_{\mathbb{C}/\mathbb{Q}}^1 = L_{\mathbb{C}/\mathbb{Q}}$, first note that the cotangent complex commutes with filtered colimits by \cite[Tag~08S9]{stacks-project}, so it suffices to show $L_{K/\QQ}=\Omega^1_{K/\QQ}$ for finitely generated subfields $K\subset\CC$. Since the cotangent complex of a finite separable extension is trivial, this reduces us to the case $K=\QQ(x_1,\dots,x_n)$ where this follows from the fact that $L_{\QQ[x_1,\dots,x_n]/\QQ}=\Omega^1_{\QQ[x_1,\dots,x_n]/\QQ}$ and $L_{\QQ(x_1,\dots,x_n)/\QQ}=L_{\QQ[x_1,\dots,x_n]/\QQ}\otimes^\bL \QQ(x_1,\dots,x_n)$ by \cite[Tag 08SF]{stacks-project}.

To obtain the bottom exact triangle, we start with the exact sequence
$$0 \longrightarrow \Omega^1_{\mathbb{C}/\mathbb{Q}} \otimes_{\mathbb{C}} \cO_X \longrightarrow \Omega^1_{X/\mathbb{Q}} \longrightarrow \Omega^1_{X/\mathbb{C}} \longrightarrow 0.$$Taking a hyperresolution $\epsilon: X_{\bullet}\to X$ and applying $R\epsilon_{\bullet *}$ gives a commutative diagram$$\begin{tikzcd} 0 \ar[r] & \Omega^1_{\mathbb{C}/\mathbb{Q}} \otimes_{\mathbb{C}} \cO_X \ar[r] \ar[d] & \Omega^1_{X/\mathbb{Q}} \ar[r] \ar[d] & \Omega^1_{X/\mathbb{C}} \ar[r] \ar[d] & 0 \\ & \Omega^1_{\mathbb{C}/\mathbb{Q}}\otimes_{\mathbb{C}} \underline{\Omega}^0_{X/\mathbb{C}} \ar[r] & \underline{\Omega}^1_{X/\mathbb{Q}} \ar[r] & \underline{\Omega}^1_{X/\mathbb{C}} \end{tikzcd}$$
On the other hand, there are natural maps $L_{X/Y}\to \mathcal{H}^0 L_{X/Y} \cong \Omega_{X/Y}^1$. Composing vertical maps with the map induced on the cotangent complex by taking $\mathcal{H}^0$, we obtain the above commutative diagram.

Then taking cones of the three vertical arrows gives the distinguished triangle
\[
\Omega^1_{\mathbb C/\mathbb Q}\otimes_{\mathbb C}C^0_{X/\mathbb C}
\longrightarrow
C^1_{X/\mathbb Q}
\longrightarrow
C^1_{X/\mathbb C}.
\]
Since $X$ is Du Bois, the natural map
\(
\mathcal O_X\to\underline{\Omega}^0_{X/\mathbb C}
\)
is a quasi-isomorphism, so $C^0_{X/\mathbb C}\simeq0$. So, we see
\[
C^1_{X/\mathbb C}\simeq C^1_{X/\mathbb Q}
\]
and hence, $\cH^0(C^1_{X/\QQ})\simeq\cH^0(C^1_{X/\CC})$. The result then follows from the fact that 
\[
\cH^0(C^1_{X/k})
\simeq
\operatorname{coker}(\Omega^1_{X/k}\longrightarrow\cH^0\underline{\Omega}^1_{X/k})
\]
for $k\subset\CC$ a subfield.
\end{proof}

\begin{proof}[{Proof of \autoref{thm:finite-quotient-k1-regular-smooth}}]
If $X$ is smooth, then it is $K_1$-regular by \cite[\S4 Theorem 3 Corollary 2 and \S7 Proposition 4.1]{Quillen1973}.

Now assume $X$ is singular and recall that $C_{X/\mathbb{Q}}^p:=\Cone(L^p_{X/\mathbb{Q}}\longrightarrow\underline{\Omega}_{X/\mathbb{Q}}^p)$. By \cite[Theorem S]{Kovacs1999}, $X$ is Du Bois. Letting $\cN K_q$ be the Zariski sheaf associated to the local Bass $NK$-groups, we have
\[
    \cN K_q\simeq\bigoplus_{j-i=-q}\cH^j(C_{X/\mathbb{Q}}^i)\otimes_\QQ t\QQ[t].
\]
by \cite{CHWW} (see also \cite[Theorem~2.10]{Shen2025}). Now $\cN K_q$ is a quasi-coherent sheaf supported on $\Sing(X)$, hence on $Z$. Let $\cI_Z$ be the ideal sheaf of $Z$, and write $\cN K_q=\varinjlim_\lambda\cF_\lambda$ as the filtered union of its finite-type quasi-coherent subsheaves \cite[Tag~01PG]{stacks-project}. Since $X$ is noetherian, each $\cF_\lambda$ is coherent, and hence $(\cI_Z)^{N_\lambda}\cF_\lambda=0$ for some $N_\lambda>0$. Thus, $\cF_\lambda$ is the pushforward of a coherent sheaf on the affine infinitesimal neighborhood $Z_{N_\lambda}$ of $Z$, so $H^p(\cF_\lambda)=0$ for $p>0$. Since cohomology commutes with filtered colimits \cite[Tag~01FF]{stacks-project}, we obtain $H^p(\cN K_q)=0$ for all $p>0$. Using the spectral sequence
\[
    E_2^{p,q}=H^p(\cN K_q)\Longrightarrow NK_{q-p}(X)
\]
we see then that
\[
    NK_1(X)\simeq H^0(\cN K_1).
\]

Since $X$ is singular, \autoref{thm:Fogarty's question} shows that $\Omega_X^1\longrightarrow\Omega_X^{[1]}$ is not surjective. Since $X$ has finite quotient singularities, $\Omega_X^{[1]}\simeq \cH^0(\underline{\Omega}_X^1)$, and so
\[
    \cH^0(C_{X/\mathbb{Q}}^1)\simeq\cH^0(C_X^1)\simeq\coker(\Omega_X^1\longrightarrow\Omega_X^{[1]})\neq0
\]
where the first isomorphism is by \autoref{l:C->Q}. Thus, $\cN K_1\neq0$ and since $\cN K_1$ is supported on the affine closed subvariety $Z$, we see $NK_1(X)=H^0(X,\cN K_1)=H^0(Z,\cN K_1)\neq0$. Therefore, $X$ is not $K_1$-regular.
\end{proof}

\section{$K$-regularity and projective cones:~proof of \autoref{thm:projective-finite-quotient-Km-regular}}

We prove \autoref{thm:projective-finite-quotient-Km-regular} by showing the following more general theorem. This theorem generalizes the construction in \cite[Theorem 4.3]{HaesemeyerWeibelLineBundles}, by Haesemeyer--Weibel.

\begin{theorem}\label{thm:Km-regular-G-torsor}
Let $G$ be an affine algebraic group, $\eta\colon P\to B$ a $G$-torsor over a smooth variety, and $Y$ a proper variety with $G$-action. Let $X=P\times^G Y$ with induced map $\pi\colon X\to B$, and let $\cV_P(D_Y^{p,j})$ be the vector bundle on $B$ associated to the $G$-representation $D_Y^{p,j}:=H^j(Y,C_Y^p)$.

If every $H^i(\Omega_B^q\otimes\cV_P(D_Y^{p,j}))=0$, then every $H^k(C_X^r)=0$. In particular, if $X$ is projective, then $X$ is $K_m$-regular for all $m$.
\end{theorem}

Since there is currently no relative Du Bois complex at the level of generality needed for \autoref{thm:Km-regular-G-torsor}, in \autoref{prop:CX-over-B-resolution-dependent} we give a technical workaround; here, $A^p_{X/B,\epsilon}$ is a resolution-dependent version of $\underline{\Omega}^p_{X/B}$.

\begin{proposition}\label{prop:CX-over-B-resolution-dependent}
With notation as in \autoref{thm:Km-regular-G-torsor}, let $\epsilon^Y_\bullet\colon Y_\bullet\to Y$ be a $G$-equivariant hyperresolution.\footnote{This exists by functorial resolution of singularities, see e.g., Proposition 3.9.1 and Theorem 3.26 of \cite{KollarResolution}.} Let $X_\bullet:=P\times^G Y_\bullet$, and $\epsilon^X_\bullet:X_\bullet\to X$ be the induced augmentation. Then

\begin{enumerate}
    \item\label{hyerpresolution} $X_\bullet\to X$ is a hyperresolution and each $\pi_\alpha\colon X_\alpha\to B$ is smooth.
    \item\label{ApXBepsilon} Letting
        \[
            A^p_{X/B,\epsilon}:=R(\epsilon^X_\bullet)_*\Omega^p_{X_\bullet/B},
        \]
        there is a natural map 
        \[
            \gamma_p\colon L^p_{X/B}\to A^p_{X/B,\epsilon}.
        \]
    \item\label{RjpiKpXB} Letting
        \[
            K^p_{X/B,\epsilon}:=\Cone(\gamma_p),
        \]
        we have
        \[
            R^j\pi_*K^p_{X/B,\epsilon}\simeq\cV_P(D_Y^{p,j}).
        \]
    \item\label{GrqCrX} $C^r_X$ admits a filtration whose associated graded is
        \[
            \Gr^q C^r_X\simeq\pi^*\Omega_B^q\otimes K^{r-q}_{X/B,\epsilon}.
        \]
\end{enumerate}
\end{proposition}

\begin{proof}
    We have a cartesian diagram
    \[
    \xymatrix{
        Y_\bullet\ar[d]_-{\epsilon^Y_\bullet} & P\times Y_\bullet\ar[d]_-{\id_P\times\epsilon^Y_\bullet}\ar[r]\ar[l]_-{\pr_{Y_\bullet}} & X_\bullet\ar[d]^-{\epsilon^X_\bullet}\\
        Y\ar[d] & P\times Y\ar[d]_-{\pr_P}\ar[l]_-{\pr_Y}\ar[r]^-{\rho} & X\ar[d]^-{\pi}\\
        \Spec\CC & P\ar[r]^-{\eta}\ar[l] & B
    }
    \]
    Since $\epsilon^Y_\bullet$ is a hyperresolution, $\id_P\times\epsilon^Y_\bullet$ is as well. Since $\eta$ is a smooth cover, and properness, surjectivity, and smoothness are smooth local properties, it follows that $\epsilon^X_\bullet$ is a hyperresolution. Furthermore, since $Y_\alpha$ is smooth, $P\times Y_\alpha\to P$ is smooth; then by descent, $\pi_\alpha\colon X_\alpha\to B$ is smooth as well. This proves (\ref{hyerpresolution}).

    We have a canonical map $L(\epsilon^X_\alpha)^*L^p_{X/B}\to L^p_{X_\alpha/B}=\Omega^p_{X_\alpha/B}$. By adjunction, this gives a canonical map
    \[
        \gamma_{p,\alpha}\colon L^p_{X/B}\to R(\epsilon^X_\alpha)_*\Omega^p_{X_\alpha/B}.
    \]
    Since this is functorial in $\alpha$, it is compatible with all face and degeneracy maps of the hyperresolution $\epsilon^X_\bullet$. Thus, we obtain our desired map
    \[
        \gamma_p\colon L^p_{X/B}\to A^p_{X/B,\epsilon}
    \]
    from (\ref{ApXBepsilon}).

    By flat base change, we have a commutative diagram
    \[
    \xymatrix{
        \rho^*L^p_{X/B}\ar[r]^-{\rho^*\gamma_{p,\alpha}}\ar[d]_-{\simeq} & \rho^*R(\epsilon^X_\alpha)_*\Omega^p_{X_\alpha/B}\ar[d]^-{\simeq}\\
        L^p_{(P\times Y)/P}\ar[r] & R(\id_P\times\epsilon^Y_\alpha)_*\Omega^p_{(P\times Y_\alpha)/P}\\
        \pr_Y^*L^p_Y\ar[u]^-{\simeq}\ar[r] & \pr_Y^*R(\epsilon^Y_\alpha)_*\Omega^p_{Y_\alpha}\ar[u]_-{\simeq}
    }
    \]
    As a result, we have isomorphisms
    \[
        \rho^*K^p_{X/B,\epsilon}\xrightarrow{\simeq} \Cone\left(L^p_{(P\times Y)/P}\to R(\id_P\times\epsilon^Y_\bullet)_*\Omega^p_{(P\times Y_\bullet)/P}\right) \xleftarrow{\simeq} \pr_Y^*C^p_Y.
    \]
    Furthermore, note that all maps in our cartesian diagram are $G$-equivariant, and hence all of the above isomorphisms in our commutative diagrams are $G$-equivariant as well. By flat base change, 
    \[
        \eta^*R^j\pi_*K^p_{X/B,\epsilon}\simeq R^j(\pr_P)_*\pr_Y^*C^p_Y\simeq \cO_P\otimes_\CC D_Y^{p,j}.
    \]
    Since this is a $G$-equivariant isomorphism, we see $R^j\pi_*K^p_{X/B,\epsilon}\simeq\cV_P(D_Y^{p,j})$, proving (\ref{RjpiKpXB}).

    Lastly, we turn to the proof of (\ref{GrqCrX}). We have an exact triangle
    \[
        \pi^*\Omega^1_B\to L_X\to L_{X/B}
    \]
    which makes $L_X$ an object of the filtered derived category. Then by \cite[Chapter V Proposition 4.2.5]{IllusieCotangentI}, we have a filtration $F^\bullet$ on $L^r_X$ whose associated graded is $\Gr^q_F L^r_X=\pi^*\Omega^q_B\otimes L^{r-q}_{X/B}$. Similarly, from the exact sequence
    \[
        0\to \pi_\alpha^*\Omega^1_B\to \Omega^1_{X_\alpha}\to \Omega^1_{X_\alpha/B}\to 0,
    \]
    we obtain a filtration $G_\alpha$ on $\Omega^r_{X_\alpha}$ whose associated graded is $\Gr^q_{G_\alpha}\Omega^r_{X_\alpha}=\pi_\alpha^*\Omega^q_B\otimes\Omega^{r-q}_{X_\alpha/B}$. Note that if $\delta\colon X_\alpha\to X_\beta$ is a face or degeneracy map of our simplicial resolution, then $\delta^*\Omega^r_{X_\beta/B}\to \Omega^r_{X_\alpha/B}$ yields a map $\delta^*G_\beta\to G_\alpha$. Hence, we obtain a filtration $G$ on $\underline{\Omega}^r_X$. From the projection formula and the fact that $\pi_\alpha=\pi\epsilon^X_\alpha$, we see
    \[
        \Gr^q_G \underline{\Omega}^r_X =R(\epsilon^X_\bullet)_*(\pi_\bullet^*\Omega^q_B\otimes\Omega^{r-q}_{X_\bullet/B}) 
        =\pi^*\Omega^q_B\otimes R(\epsilon^X_\bullet)_*\Omega^{r-q}_{X_\bullet/B}
        =\pi^*\Omega^q_B\otimes A^{r-q}_{X/B,\epsilon}.
    \]
    Note that we have a morphism of exact triangles
    \[
        \xymatrix{
        (\epsilon^X_\alpha)^*\pi^*\Omega^1_B\ar[r]\ar[d]_-{\simeq} & L(\epsilon^X_\alpha)^*L_X\ar[r]\ar[d] & L(\epsilon^X_\alpha)^*L_{X/B}\ar[d]\\
        \pi_\alpha^*\Omega^1_B\ar[r] & \Omega^1_{X_\alpha}\ar[r] & \Omega^1_{X_\alpha/B}
        }
    \]
    Thus, the canonical map $L^r_X\to\underline{\Omega}^r_X$ respects the filtrations $F$ and $G$, and so we obtain an induced filtration on $C^r_X$ with
    \[
        \Gr^q C^r_X\simeq\Cone\left(\pi^*\Omega^q_B\otimes L^{r-q}_{X/B}\xrightarrow{\id\otimes\gamma_{r-q}} \pi^*\Omega^q_B\otimes A^{r-q}_{X/B,\epsilon}\right)\simeq \pi^*\Omega^q_B\otimes K^{r-q}_{X/B,\epsilon},
    \]
    which shows (\ref{GrqCrX}).
\end{proof}

\begin{proof}[{Proof of \autoref{thm:Km-regular-G-torsor}}]
    By \autoref{prop:CX-over-B-resolution-dependent}(\ref{GrqCrX}), we know
    \[
        \Gr^q C^r_X=\pi^*\Omega^q_B\otimes K^{r-q}_{X/B,\epsilon}
    \]
    and to prove $R\Gamma(C^r_X)=0$, it suffices to show vanishing of $R\Gamma(\pi^*\Omega^q_B\otimes K^{r-q}_{X/B,\epsilon})=R\Gamma(\Omega^q_B\otimes R\pi_*K^{r-q}_{X/B,\epsilon})$. Then from the spectral sequence
    \[
        E_2^{ij}=H^i(\Omega^q_B\otimes R^j\pi_*K^{r-q}_{X/B,\epsilon})\Rightarrow H^{i+j}(\Omega^q_B\otimes R\pi_*K^{r-q}_{X/B,\epsilon}),
    \]
    it suffices to show $H^i(\Omega^q_B\otimes R^j\pi_*K^{r-q}_{X/B,\epsilon})=0$. This follows from \autoref{prop:CX-over-B-resolution-dependent}(\ref{RjpiKpXB}). Thus, we have shown $R\Gamma(C^r_{X/\CC})=0$. If $X$ is projective, \cite[Theorem~7.1]{Shen2025} therefore shows that $X$ is $K_m$-regular for all $m$.
\end{proof}

We next show the following application giving a class of examples which are $K_m$-regular for all $m$.

\begin{corollary}\label{cor:Km-regular-cone}
    Let $Y'\subset\bP(V)$ be a smooth connected projective variety, and let $Y\subset\bP(V\oplus\CC)$ be its projective cone. Let $e\in\ZZ^+$ and let $\bG_m$ act on $Y$ by $\lambda(v:c)=(v:\lambda^e c)$. If $A$ is an abelian variety and $\cL\in\Pic^0(A)$ is a non-torsion line bundle with associated $\bG_m$-torsor $P\to A$, then $X=P\times^{\bG_m}Y$ is projective and $K_m$-regular for all $m$.
\end{corollary}

\begin{proof}
    Since $Y\subset\bP(V\oplus\CC)$ is a $\bG_m$-equivariant closed immersion, we have a closed immersion $X=P\times^{\bG_m}Y\subset\bP_A(P\times^{\bG_m}(V\oplus\CC))$. Thus, $X$ is projective.

    Let $D^{p,j}_{Y,w}$ denote the weight $w$ piece of $D^{p,j}_Y$. Then $\cV_P(D^{p,j}_Y)=\bigoplus_w(D^{p,j}_{Y,w}\otimes_\CC\cL^{\otimes w})$. Since $\Omega^q_A$ is trivial, by \autoref{thm:Km-regular-G-torsor}, we need only show $D^{p,j}_{Y,w}\otimes_\CC H^i(\cL^{\otimes w})=H^i(D^{p,j}_{Y,w}\otimes_\CC\cL^{\otimes w})=0$ for all $i$. Note that if $w\neq0$, then $\cL^{\otimes w}$ is a non-trivial degree $0$ line bundle, so \cite[p.~76 (vii)]{MumfordAbelianVarieties} tells us $H^i(\cL^{\otimes w})=0$.

    To finish the proof, we therefore need only show $D^{p,j}_{Y,0}=0$. Note that $C^p_Y$ is supported on the vertex $v$ of the cone, so we may restrict attention to the affine cone $Z=\Spec R$ over $Y'$. We may write $R=\bigoplus_{n\geq0}R_n$, where $R_0=\CC$ and $R$ has only positive weights coming from the $\bG_m$-action. Let $P_\bullet\to R$ be a graded simplicial polynomial resolution. Since $R_0=\CC$, we may choose each $P_n=\CC[x_{n,\alpha}]$ with $x_{n,\alpha}$ of positive weight. Then every homogeneous element of $\Omega^1_{P_n}\otimes_{P_n}R\simeq\bigoplus_\alpha R dx_{n,\alpha}$ has positive weight, and so $(L^p_Z)_0=0$ for $p>0$. For $p=0$, we have canonical isomorphisms $(L^0_Z)_0\simeq R_0=\CC$. 

    Next consider the Du Bois complex $\underline{\Omega}^p_Z$. Let $f\colon\widetilde{Z}\to Z$ be the blow-up of $Z$ at $v$, and let $g\colon\widetilde{Z}\to E=Y'$ be the projection to the exceptional divisor. Note that $\widetilde{Z}$ is smooth. Consider the short exact sequence
    \[
        0\to g^*\Omega^1_{Y'}\to\Omega^1_{\widetilde{Z}}\to\Omega^1_{\widetilde{Z}/Y'}\to0.
    \]
    Since $\Omega^1_{\widetilde{Z}/Y'}=g^*\cO(1)$, taking exterior powers we have another short exact sequence given by
    \[
        0\to g^*\Omega^p_{Y'}\to\Omega^p_{\widetilde{Z}}\to g^*\Omega^{p-1}_{Y'}(1)\to0.
    \]
    Since $\widetilde{Z}=\underline{\Spec}_{Y'}\Sym\cO(1)$, we see $g_*\cO_{\widetilde{Z}}=\bigoplus_{n\geq0}\cO(n)$, and so
    \[
        0\to \bigoplus_{n\geq0}\Omega^p_{Y'}(n)\to g_*\Omega^p_{\widetilde{Z}}\to \bigoplus_{n\geq0}\Omega^{p-1}_{Y'}(n+1)\to0
    \]
    is exact. We therefore have an isomorphism on weight $0$ pieces
    \[
        \beta\colon\Omega^p_{Y',0}\xrightarrow{\simeq} g_*\Omega^p_{\widetilde{Z},0}.
    \]
    Next, since $f$ is an equivariant map, we have an equivariant exact triangle
    \[
        R\Gamma(\underline{\Omega}^p_Z)\to R\Gamma(\underline{\Omega}^p_{\widetilde{Z}})\oplus\Omega^p_v\to R\Gamma(\underline{\Omega}^p_{Y'}).
    \]
    When $p>0$, by considering the weight zero piece, this becomes an exact triangle
    \[
        R\Gamma(\underline{\Omega}^p_Z)_0\to R\Gamma(\underline{\Omega}^p_{Y'})\xrightarrow{\gamma} R\Gamma(\underline{\Omega}^p_{Y'}).
    \]
    The map $\gamma$ is induced by inclusion $\iota\colon Y'\to\widetilde{Z}$. Since $g\iota=\id_{Y'}$, we see $\gamma$ is an isomorphism, namely the inverse of $\beta$. Thus, 
    \[
        R\Gamma(\underline{\Omega}^p_Z)_0=0.
    \]
    When $p=0$, the exact triangle yields
    \[
        R\Gamma(\underline{\Omega}^0_Z)_0\to R\Gamma(\cO_{Y'})\oplus\CC\xrightarrow{\delta} R\Gamma(\cO_{Y'})
    \]
    where $\delta(y',c)=y'-c$. Thus, $R\Gamma(\underline{\Omega}^0_Z)_0\simeq\CC$ and we see $R\Gamma(C^p_Z)_0=0$ as desired.
\end{proof}

Finally we turn to \autoref{thm:projective-finite-quotient-Km-regular}.

\begin{proof}[{Proof of \autoref{thm:projective-finite-quotient-Km-regular}}]
Let $Y'=\nu_d(\bP^1)\subset\bP^d$ be the $d$-th Veronese embedding, and let $Y$ be the projective cone over $Y'$. Then $Y=\bP(1,1,d)\simeq\bP^2/\mu_d$. 

Let $P\to A$ be the $\bG_m$-torsor associated to $\cL$. Note that if $\bG_m$ acts on $\bA^3$ with weights $(0,0,1)$, then $Z:=\bP_A(\cO^{\oplus2}\oplus\cL)=\bP_A(P\times^{\bG_m}\bA^3)=P\times^{\bG_m}\bP^2$. Then we see the $\mu_d$-action on $Z$ lifts to an action on $P\times\bA^3$ where it acts via the $\bG_m$-action on $\bA^3$ through the inclusion $\mu_d\subset\bG_m$. Since the $\bG_m$-action on $\bA^3$ and this $\mu_d$-action commute, we see
\[
    Z/\mu_d=((P\times\bP^2)/\bG_m)/\mu_d=P\times^{\bG_m}(\bP^2/\mu_d).
\]
Thus, by \autoref{cor:Km-regular-cone}, we see $Z/\mu_d$ is $K_m$-regular for all $m$. Furthermore, on an open affine cover $U_i$ of $A$ where $\cL$ is trivial, we see $Z/\mu_d$ is given by $U_i\times(\bP^2/\mu_d)=U_i\times\bP(1,1,d)$. Since $d\geq3$ this is not lci.
\end{proof}

\bibliographystyle{plain}
\bibliography{references}

\end{document}